\documentclass{amsart}

\usepackage{graphicx,psfrag,url}

\newtheorem{proposition}{Proposition}[section]
\newtheorem{theorem}[proposition]{Theorem}
\newtheorem{lemma}[proposition]{Lemma}

\newtheorem{corollary}[proposition]{Corollary}

\DeclareMathOperator{\vol}{Vol}

\title{A strong triangle inequality in hyperbolic geometry}

\author[C. Bir\'o]{Csaba Bir\'o}
\address{Department of Mathematics, University of Louisville, Louisville, KY 40292}
\email{csaba.biro@louisville.edu}

\author[R. C. Powers]{Robert C. Powers}
\address{Department of Mathematics, University of Louisville, Louisville, KY 40292}
\email{robert.powers@louisville.edu}

\begin{document}

\begin{abstract}
For a triangle in the hyperbolic plane, let $\alpha,\beta,\gamma$ denote the angles opposite the sides $a,b,c$, respectively. Also, let $h$ be the height of the altitude to side $c$. Under the assumption that $\alpha,\beta, \gamma$ can be chosen uniformly in the interval $(0,\pi)$ and it is given that $\alpha+\beta+\gamma<\pi$, we show that the strong triangle inequality $a + b > c + h$ holds approximately 79\% of the time. To accomplish this, we prove a number of theoretical results to make sure that the probability can be computed to an arbitrary precision, and the error can be bounded.
\end{abstract}

\maketitle

\section{Introduction}

It is well known that the Euclidean and hyperbolic planes satisfy the triangle inequality. What is less known is that in many cases a stronger triangle inequality holds. Specifically,
\begin{equation} \label{eq:sti}
a + b > c + h
\end{equation}
where $a,b,c$ are the lengths of the three sides of the triangle and $h$ is the height of the altitude to side $c$. We refer to inequality (\ref{eq:sti}) as the {\em strong triangle inequality} and note that this inequality depends on which side of the triangle is labeled $c$.

The strong triangle inequality was first introduced for the Euclidean plane by Bailey and Bannister in \cite{Bai-Ban-97}. They proved, see also Klamkin \cite{Kla-98}, that inequality (\ref{eq:sti}) holds for all Euclidean triangles if $\gamma < \arctan \left (\frac{24}{7}\right)$ where $\gamma$ is the angle opposite side $c$. Bailey and Bannister also showed that $a + b = c + h$ for any Euclidean isosceles triangle such that $\gamma = \arctan \left (\frac{24}{7}\right)$ and $\gamma$ is the unique largest angle of the triangle. We let $B = \arctan \left (\frac{24}{7}\right)$ and refer to $B$ as the Bailey-Bannister bound.

In 2007, Baker and Powers \cite{Bak-Pow-07} showed that the strong triangle inequality holds for any hyperbolic triangle if $\gamma \leq \Gamma$ where $\Gamma$ is the unique root of the function
\[ f(\gamma) = -1 - \cos \gamma + \sin \gamma + \sin \frac{\gamma}{2} \sin \gamma\]
in the interval $[0, \frac{\pi}{2}]$. It turns out that $B \approx 74^{\circ}$ and $\Gamma \approx 66^{\circ}$ leading to roughly an $8^{\circ}$ difference between the Euclidean and hyperbolic bounds. It appears that the strong triangle inequality  holds more often in the Euclidean plane than in the hyperbolic plane.

Let $\alpha$ and $\beta$ denote the angles opposite the sides $a$ and $b$, respectively. Under the assumption that the angles $\alpha$ and $\beta$ can be chosen uniformly in the interval $(0, \pi)$ and $\alpha + \beta < \pi$, Fa{\u\i}ziev et al. \cite{Fai-Pow-Sah-13} showed the strong triangle inequality holds in the Euclidean plane approximately 69\% of the time. In addition, they asked how this percentage will change when working with triangles in the hyperbolic plane. In this paper, we answer this question by showing that the strong triangle inequality holds approximately 79\% of the time. Moreover, we show that the stated probability can be computed to an arbitrary precision and that the error can be bounded.

Unless otherwise noted, all geometric notions in this paper are on the
hyperbolic plane. Since our problem is invariant under scaling, we will assume
that the Gaussian curvature of the plane is $-1$. We will use the notations $a$, $b$, $c$, $h$, $\alpha$, $\beta$, $\gamma$ for sides, height, and angles of a given triangle. (See Figure~\ref{fig:triangle}.)
\begin{figure}
\psfrag{a}{$a$}
\psfrag{b}{$b$}
\psfrag{c}{$c$}
\psfrag{al}{$\alpha$}
\psfrag{be}{$\beta$}
\psfrag{ga}{$\gamma$}
\psfrag{h}{$h$}
\includegraphics[scale=0.8]{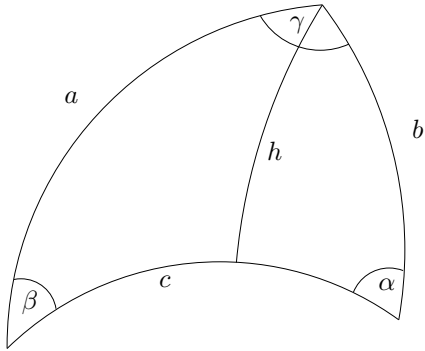}
\caption{A hyperbolic triangle\label{fig:triangle}}
\end{figure}
We will extensively use
hyperbolic trigonometric formulas such as the law of sines and the two versions of the
law of cosines. We refer the reader to Chapter 8 in \cite{Sta-TPHP} for a list of these various formulas.

\section{Simple observations}\label{section:observations}

In this section we mention a few simple, but important observations about the
main question.

\begin{proposition}\label{proposition:greatest}
If $\gamma$ is not the unique greatest angle in a triangle, then the strong triangle
inequality holds.
\end{proposition}

\begin{proof}
Suppose that $\gamma$ is not the greatest angle, say, $\alpha\geq\gamma$. Then
$a\geq c$, so $a+b\geq c+b\geq c+h$. Equality could only hold, if
$\alpha=\gamma=\pi/2$, which is impossible.
\end{proof}

\begin{proposition}\label{proposition:obtuse}
If $\gamma\geq\pi/2$, then the strong triangle inequality does not hold.
\end{proposition}

We will start with a lemma that is interesting in its own right.

\begin{lemma}\label{lemma:area}
In every triangle the following equation holds.
\[
\sinh c\sinh h=\sinh a\sinh b\sin\gamma
\]
\end{lemma}

Note that in Euclidean geometry the analogous theorem would be the statement
that $ch=ab\sin\gamma$, which is true by the fact that both sides of the
equation represent twice the area of the triangle. Interestingly, in hyperbolic
geometry, the sides of the corresponding equation do not represent
the area of the triangle.

\begin{proof}
By the law of sines,
\[
\frac{\sinh b}{\sin\beta}=\frac{\sinh c}{\sin\gamma},
\]
so
\[
\sinh c=\frac{\sinh b\sin\gamma}{\sin\beta}.
\]
By right triangle trigonometry, $\sinh h=\sinh a\sin\beta$. Multiplying these
equations, the result follows.
\end{proof}

\begin{proof}[Proof of Proposition~\ref{proposition:obtuse}]
Note that $a+b>c+h$ if and only if $\cosh(a+b)>\cosh(c+h)$. Using the addition
formula for $\cosh$, then the fact that $\cosh h\geq 1$ and $\cosh c\geq 0$,
and then the law of cosines, in this order, we get
\begin{multline*}
\cos(c+h)
=\cosh c\cosh h+\sinh c\sinh h
\geq\cosh c+\sinh c\sinh h\\
=\cosh a\cosh b-\sinh a\sinh b\cos\gamma+\sinh a\sinh b\sin\gamma\\
=\cosh a\cosh b+\sinh a\sinh b(\sin\gamma-\cos\gamma)
\end{multline*}
Notice that $\sin\gamma-\cos\gamma\geq 1$ if $\pi/2\leq\gamma\leq\pi$. So
\begin{multline*}
\cosh a\cosh b+\sinh a\sinh b(\sin\gamma-\cos\gamma)\\
\geq\cosh a\cosh b+\sinh a\sinh b
=\cosh(a+b).
\end{multline*}
\end{proof}

\section{Converting angles to lengths}

Since the angles of a hyperbolic triangle uniquely determine the triangle, it
is possible to rephrase the condition $a+b>c+h$ with $\alpha,\beta,\gamma$. In
what follows, our goal is find a function $f(\alpha,\beta,\gamma)$, as simple as
possible, such that $a+b>c+h$ if and only if $f(\alpha,\beta,\gamma)>0$. Following
Proposition~\ref{proposition:greatest} and
Proposition~\ref{proposition:obtuse}, in the rest of the section we will
assume that $\gamma<\pi/2$ is the greatest angle of the triangle.

The following lemma is implicit in \cite{Bak-Pow-07}. We
include the proof for completeness.

\begin{lemma}\label{lemma:BakerPowers}
A triangle satisfies the strong triangle inequality if and only if
\[
\frac{\cos\alpha\cos\beta+\cos\gamma}{\cos\gamma+1-\sin\gamma}-1<\cosh h.
\]
Furthermore, the formula holds with equality if and only if $a+b=c+h$.
\end{lemma}

\begin{proof}
Recall that $a+b>c+h$ if and only if $\cosh(a+b)-\cosh(c+h)>0$. Using the $\cosh$ addition formula and the law of
cosines on $\cosh c$, we have
\begin{multline*}
\cosh(a+b)-\cosh(c+h)\\
=\cosh a\cosh b+\sinh a\sinh b-\cosh c\cosh h-\sinh c\sinh h\\
=\cosh c+\sinh a\sinh b\cos\gamma+\sinh a\sinh b-\cosh c\cosh h-\sinh c\sinh
h\\
=\cosh c(1-\cosh h)+\sinh a\sinh b(\cos\gamma+1)-\sinh c\sinh h
\end{multline*}

By Lemma~\ref{lemma:area},
\begin{multline*}
\cosh c(1-\cosh h)+\sinh a\sinh b(\cos\gamma+1)-\sinh c\sinh h\\
=\cosh c(1-\cosh h)+\sinh a\sinh b(\cos\gamma+1)-\sinh a\sinh b\sin\gamma\\
=\frac{\sinh a\sinh b}{1+\cosh h}\left[\cosh c\frac{1-\cosh^2 h}{\sinh a\sinh
b}+(1+\cosh h)(\cos\gamma+1-\sin\gamma)\right]
\end{multline*}
By the fact that $\sinh h=\sinh b\sin\alpha=\sinh a\sin\beta$, we have
\[
\frac{1-\cosh^2 h}{\sinh a\sinh b}
=\frac{-\sinh^2 h}{\sinh a\sinh b}
=\frac{-\sinh b\sin\alpha\cdot\sinh a\sin\beta}{\sinh a\sinh b}
=-\sin\alpha\sin\beta,
\]
so, using the dual form of the law of cosines,
\begin{multline*}
\cos(a+b)-\cosh(c+h)\\
=\frac{\sinh a\sinh b}{1+\cosh h}\left[-\cosh c\sin\alpha\sin\beta+(1+\cosh
h)(\cos\gamma+1-\sin\gamma)\right]\\
=\frac{\sinh a\sinh b}{1+\cosh
h}\left[-(\cos\alpha\cos\beta+\cos\gamma)+(1+\cosh
h)(\cos\gamma+1-\sin\gamma)\right].
\end{multline*}

Since $\frac{\sinh a\sinh b}{1+\cosh h}>0$, we have the strong triangle
inequality holds, if and only if
\[
\cos\alpha\cos\beta+\cos\gamma<(1+\cosh
h)(\cos\gamma+1-\sin\gamma),
\]
and the result follows.

A minor variation of the proof shows the case of equality.
\end{proof}

\begin{lemma}\label{lemma:nonnegative}
For all triangles with $\gamma > \max\{\alpha, \beta\}$,
\[
\frac{\cos\alpha\cos\beta+\cos\gamma}{\cos\gamma+1-\sin\gamma}>1.
\]
\end{lemma}

\begin{proof}
Without loss of generality,
$0<\alpha\leq\beta<\gamma<\pi/2$. Then
\begin{gather*}
0>\sin\gamma(\sin\gamma-1)=\sin^2\gamma-\sin\gamma>\sin^2\beta-\sin\gamma\\
\cos^2\beta>\cos^2\beta+\sin^2\beta-\sin\gamma=1-\sin\gamma\\
\cos^2\beta+\cos\gamma>1-\sin\gamma+\cos\gamma,
\end{gather*}
so
\[
\frac{\cos^2\beta+\cos\gamma}{1-\sin\gamma+\cos\gamma}>1.
\]
Since $0<\cos\beta\leq\cos\alpha$, we have
\[
\frac{\cos^2\beta+\cos\gamma}{1-\sin\gamma+\cos\gamma}
\leq\frac{\cos\alpha\cos\beta+\cos\gamma}{1-\sin\gamma+\cos\gamma},
\]
and the results follows.
\end{proof}

By Lemma~\ref{lemma:BakerPowers} and Lemma~\ref{lemma:nonnegative}, we can
conclude that the strong triangle
inequality holds if and only if
\begin{equation}
\left(\frac{\cos\alpha\cos\beta+\cos\gamma}{\cos\gamma+1-\sin\gamma}-1\right)^2
<\cosh^2 h.\label{eq:1}
\end{equation}
Using the law of cosines,
\begin{multline}
\cosh^2 h
=\sinh^2 h+1
=\sin^2\beta\sinh^2 a+1
=\sin^2\beta(\cosh^2 a-1)+1\\
=\sin^2\beta\left(\frac{\cos\beta\cos\gamma+\cos\alpha}{\sin\beta\sin\gamma}\right)^2-\sin^2\beta+1\\
=\cos^2\beta+\left(\frac{\cos\beta\cos\gamma+\cos\alpha}{\sin\gamma}\right)^2.\label{eq:2}
\end{multline}
Equations (\ref{eq:1}) and (\ref{eq:2}) together imply the following statement.

\begin{lemma}\label{lemma:f}
The strong triangle inequality holds if and only if
\[
f(\alpha,\beta,\gamma)=
\cos^2\beta+\left(\frac{\cos\beta\cos\gamma+\cos\alpha}{\sin\gamma}\right)^2-
\left(\frac{\cos\alpha\cos\beta+\cos\gamma}{\cos\gamma+1-\sin\gamma}-1\right)^2
>0
\]
\end{lemma}

Notes:
\begin{enumerate}
\item $f(\alpha,\beta,\gamma)=0$ if and only if $a+b=c+h$. The proof of this is
a minor variation of that of Lemma~\ref{lemma:f}.
\item $f(\alpha,\beta,\gamma)$ is symmetric in $\alpha$ and $\beta$. This is
obvious from the geometry, but it is also not hard to prove directly.
\item $f(\alpha,\beta,\gamma)$ is quadratic in $\cos\alpha$ and $\cos\beta$.
\item $f(\alpha,\beta,\gamma)$ is not monotone in either $a+b-c-h$ or
in $\cosh(a+b)-\cosh(c+h)$. Therefore it is not directly useful for studying the
difference of the two sides in the strong triangle inequality.
\end{enumerate}

Also note that it is fairly trivial to write down the condition $a+b>c+h$ with
an inequality involving only $\alpha$, $\beta$, and $\gamma$. Indeed, one can
just use the law of cosines to compute $a$, $b$, and $c$ from the angles, and
some right triangle trigonometry to compute $h$. But just doing this simple
approach will result in a formidable formula with inverse trigonometric
functions and square roots. Even if one uses the fact that the condition
is equivalent to $\cosh(a+b)>\cosh(c+h)$, the resulting naive formula is
hopelessly complicated, and certainly not trivial to solve for $\alpha$ and
$\beta$. Therefore, the importance and depth of Lemma~\ref{lemma:f} should not
be underestimated.

\section{Computing probabilities}

Motivated by the original goal of computing the probability that the strong
triangle inequality holds in hyperbolic geometry, we need to clarify first under
what model we compute this probability.

In hyperbolic geometry there exists a triangle for arbitrarily chosen angles,
provided that their sum is less than $\pi$. So it is natural to choose the
three angles independently uniformly at random in $(0,\pi)$, and then aim to
compute the probability that the strong triangle inequality holds, given that
the sum of the chosen angles is less than $\pi$.

Of course, the computation can be reduced to a computation of volumes.
Let
\[
F=\{(\alpha,\beta,\gamma)\in(0,\pi)^3:\alpha+\beta+\gamma<\pi\text{, }\max\{\alpha,\beta\}<\gamma<\pi/2\},
\]
and let
\[
S=\{(\alpha,\beta,\gamma)\in F:f(\alpha,\beta,\gamma)>0\}.
\]
Since $f$ is continuous when $0<\gamma<\pi/2$, and $S$ is the level set of $f$
(within $F$), $S$ is measurable, so its volume is well-defined. The desired
probability is then
\[
\frac{\vol(S)}{\pi^3/6},
\]
where the denominator is the volume of the tetrahedron for which
$\alpha+\beta+\gamma<\pi$.

So it remains to compute the volume of $S$. Fix $0<\gamma<\pi/2$, and let
\begin{align*}
P_\gamma&=\{(\alpha,\beta):(\alpha,\beta,\gamma)\in F\text{ and
}f(\alpha,\beta,\gamma)>0\}\\
N_\gamma&=\{(\alpha,\beta):(\alpha,\beta,\gamma)\in F\text{ and
}f(\alpha,\beta,\gamma)<0\}\\
Z_\gamma&=\{(\alpha,\beta):(\alpha,\beta,\gamma)\in F\text{ and
}f(\alpha,\beta,\gamma)=0\}.
\end{align*}
(See Figures~\ref{fig:1-2} and \ref{fig:1-3} for illustration for $\gamma=1.2$ and $\gamma=1.3$ respectively.)
It is clear that
\[
\vol(S)=\int_0^{\pi/2}\mu(P_\gamma)\,d\gamma,
\]
where $\mu$ is the 2-dimensional Lebesgue measure.

It is not hard to see why it will be useful for us to solve the equation
$f(\alpha,\beta,\gamma)=0$: it will provide a description of the set
$Z_\gamma$, which will help us analyze the sets $P_\gamma$, and $N_\gamma$.
This is easy, because $f$ is quadratic in $\cos\beta$. The following extremely
useful lemma shows that at most one of the quadratic solutions will lie in $F$.

\begin{lemma}\label{lemma:whichroot}
Let $(\alpha,\beta,\gamma)\in F$
such that $f(\alpha,\beta,\gamma)=0$. Let
\begin{align*}
a&=\csc^2\gamma-\left(\frac{\cos\alpha}{\cos\gamma+1-\sin\gamma}\right)^2\\
b&=\frac{\cos\alpha(\cos\gamma+1)}{\sin^2\gamma}\\
c&=\left(\frac{\cos\alpha}{\sin\gamma}\right)^2-\left(\frac{1-\sin\gamma}{\cos\gamma+1-\sin\gamma}\right)^2.
\end{align*}
Then
\[
\cos\beta=\frac{-b-\sqrt{b^2-4ac}}{2a}.
\]
\end{lemma}
\begin{proof}
By tedious, but simple algebra one can see that $f(\alpha,\beta,\gamma)=0$ if and only if
$a\cos^2\beta+b\cos\beta+c=0$. To see the result, we will show that if
$(\alpha,\beta,\gamma)\in F$, then $(-b+\sqrt{b^2-4ac})/(2a)<0$. We will
proceed by showing that for $(\alpha,\beta,\gamma)\in F$, we have $b>0$, and
$c>0$. The former is trivial. For the latter, here follows the sequence of
implied inequalities.
\begin{gather*}
1+\cos\gamma>\sin\gamma(1+\cos\gamma)=\sin\gamma+\sin\gamma\cos\gamma\\
\cos^2\gamma+\sin^2\gamma+\cos\gamma>\sin\gamma+\sin\gamma\cos\gamma\\
\cos\gamma(\cos\gamma+1-\sin\gamma)>\sin\gamma(1-\sin\gamma)\\
\frac{\cos\alpha}{\sin\gamma}\geq\frac{\cos\gamma}{\sin\gamma}>\frac{1-\sin\gamma}{\cos\gamma+1-\sin\gamma}
\end{gather*}
Squaring both sides will give $c>0$.

We have shown that $b,c>0$. If $a=0$, then $\cos\beta=-c/b<0$, and that is
inadmissible. If $a>0$, then $b^2-4ac<b^2$, so $(-b+\sqrt{b^2-4ac})/(2a)<0$, and
similarly, if $a<0$, then $b^2-4ac>b^2$, so $(-b+\sqrt{b^2-4ac})/(2a)<0$ again.
\end{proof}

Recall that a set $R\subseteq\mathbb{R}\times\mathbb{R}$ is called a
\emph{function}, if for all $x\in\mathbb{R}$ there is at most one
$y\in\mathbb{R}$ such that $(x,y)\in R$. Also $R^{-1}=\{(y,x):(x,y)\in R\}$.
$R$ is \emph{symmetric}, if $R=R^{-1}$. The domain of a function $R$ is the set
$\{x:\exists y, (x,y)\in R\}$.

So far we have learned the following about $Z_\gamma$.
\begin{itemize}
\item $Z_\gamma$ is a function (by Lemma~\ref{lemma:whichroot}).
\item $Z_\gamma$ is symmetric.
\item $Z_\gamma$ is injective (that is $Z_\gamma^{-1}$ is function).
\item $\mu(Z_\gamma)=0$.
\end{itemize}
The last fact follows, because $Z_\gamma$ is closed, and hence, it is
measurable.

Therefore the computation may be reduced to that of $\mu(N_\gamma)$,
which will turn out to be more convenient.

Let $\gamma$ be fixed, and let $0<\alpha<\gamma$. We will say that $\alpha$ is
\emph{all-positive}, if for all $\beta$ we have $f(\alpha,\beta,\gamma)\geq 0$.
Similarly, $\alpha$ is \emph{all-negative}, if for all $\beta$,
$f(\alpha,\beta,\gamma)\leq 0$. If there is a $\beta'$ such that
$f(\alpha,\beta',\gamma)=0$, then there are two possibilities: if for all
$\beta<\beta'$, we have $f(\alpha,\beta,\gamma)<0$, and for all $\beta>\beta'$,
we have $f(\alpha,\beta,\gamma)>0$, then we will say $\alpha$ is
\emph{negative-positive}. If it's the other way around, we will say $\alpha$ is
\emph{positive-negative}.

We will use the function notation $z(\alpha)=\beta$, when
$(\alpha,\beta)\in Z_\gamma$; $z(\alpha)$ is undefined if
$\alpha$ is not in the domain of $Z_\gamma$. When we want to emphasize the
dependence on $\gamma$, we may write $z_\gamma(\alpha)$ for $z(\alpha)$.

\begin{lemma}\label{lemma:formula}
If $z_\gamma$ is defined at $\alpha$, then
\[
z_\gamma(\alpha)=\cos^{-1}\left(\frac{-b-\sqrt{b^2-4ac}}{2a}\right),
\]
where $a,b,c$ are as in Lemma~\ref{lemma:whichroot}.
\end{lemma}

\begin{proof}
This is direct consequence of Lemma~\ref{lemma:whichroot}.
\end{proof}

Our next goal is to extend the set $F$ as follows:
\[ \overline{F} =\{(\alpha,\beta,\gamma)\in(0,\pi)^3:\alpha+\beta+\gamma\leq\pi\text{, }\max\{\alpha,\beta\} \leq\gamma<\pi/2\}.
\]
So we are extending $F$ by considering cases where $\alpha + \beta + \gamma = \pi$ and where $\max\{\alpha, \beta\} = \gamma$.
For fixed $\gamma$ such that $0 < \gamma < \pi/2$, the collection $\{P_\gamma,N_\gamma,Z_\gamma\}$ is extended by letting
\begin{align*}
\overline{P}_\gamma&= P_\gamma\ \cup\ \{(\alpha,\beta):(\alpha,\beta,\gamma)\in \overline{F}\setminus F\text{ and } a + b > c + h\}\\
\overline{N}_\gamma&= N_\gamma\ \cup\ \{(\alpha,\beta):(\alpha,\beta,\gamma)\in \overline{F}\setminus F\text{ and } a + b < c + h\}\\
\overline{Z}_\gamma&= Z_\gamma\ \cup\ \{(\alpha,\beta):(\alpha,\beta,\gamma)\in \overline{F}\setminus F\text{ and } a + b = c + h\}.
\end{align*}
We will show that with this extension, sequences of points
entirely outside of $P_\gamma$ can not converge to a point in $\overline{P}_\gamma$, and
similarly for $N_\gamma$. But first, we need a lemma that, in
a way, formalizes the well-known intuition that infinitesimally small
hyperbolic triangles are becoming arbitrarily similar to Euclidean triangles.

\begin{lemma}\label{lemma:infinitesimal}
Let $\{(\alpha_i,\beta_i,\gamma_i)\}_{i=1}^\infty$ be a sequence in
$\mathbb{R}^3$, such that $(\alpha_i,\beta_i,\gamma_i)\to(\alpha,\beta,\gamma)$
with $\alpha_i+\beta_i+\gamma_i<\pi$ for all $i$, and $\alpha+\beta+\gamma=\pi$. Let $a_i,b_i,c_i$ be the sides of the hyperbolic
triangle determined by $\alpha_i,\beta_i,\gamma_i$, and let $h_i$ be the height
corresponding to $c_i$. Furthermore, consider the class of similar Euclidean
triangles with angles $\alpha,\beta,\gamma$, and let $a,b,c$ be the sides of an
element of this class, and let $h$ be the height corresponding to $c$. Then
\[
\lim_{i\to\infty}\frac{a_i+b_i-c_i}{h_i}=\frac{a+b-c}{h}.
\]
\end{lemma}

\begin{proof}
First we will prove that $a_i/b_i\to a/b$. By the law of sines for both the hyperbolic and Euclidean planes,
\[
\frac{\sinh a_i}{\sinh
b_i}=\frac{\sin\alpha_i}{\sin\beta_i}\to\frac{\sin\alpha}{\sin\beta}=\frac ab.
\]
Since $a_i,b_i\to 0$, $\lim(\sinh a_i/\sinh b_i)=\lim(a_i/b_i)$, and the claim
follows.

Note that applying this to various triangles formed by the height and the
sides, this also implies $a_i/h_i\to a/h$, and $b_i/h_i\to b/h$. To see that
$c_i/h_i\to c/h$, just observe that $c_i/h_i=(c_i/b_i)(b_i/h_i)$.
\end{proof}

\begin{corollary}\label{corollary:nojump}
Let $\gamma\in(0,\pi/2)$.
Let $\{(\alpha_i,\beta_i)\}_{i=1}^\infty$ be a sequence in $\mathbb{R}^2$ such
that $(\alpha_i,\beta_i)\to (\alpha,\beta)$ with $(\alpha_i,\beta_i,\gamma)\in
F$ and $(\alpha,\beta,\gamma)\in \overline{F}$. Then $(\alpha_i,\beta_i)\not\in P_\gamma$ implies
$(\alpha,\beta)\not\in \overline{P}_\gamma$, and $(\alpha_i,\beta_i)\not\in N_\gamma$
implies $(\alpha,\beta)\not\in \overline{N}_\gamma$.
\end{corollary}

\begin{proof}
If $(\alpha,\beta,\gamma)\in F$, then this is a direct consequence of the
continuity of $f$. If $(\alpha,\beta,\gamma)$ belongs to $\overline{F} \setminus F$, then $\alpha=\gamma$ or $\beta=\gamma$ or $\alpha+\beta+\gamma=\pi$. In the first two cases, all distances in the
triangles determined by the angles $\alpha_i,\beta_i,\gamma_i$ converge to the
corresponding distances in the limiting isosceles hyperbolic triangle
determined by the angles $\alpha,\beta,\gamma$. Finally, if
$\alpha+\beta+\gamma=\pi$, then this is a consequence of
Lemma~\ref{lemma:infinitesimal} with the observation that the strong triangle
inequality holds if and only if $\frac{a+b-c}{h}>1$.
\end{proof}

Recall the notations $\Gamma$ for the Baker--Powers constant, and $B$ for the
Bailey--Bannister constant. We will use the following lemma which was proven by Baker and Powers \cite{Bak-Pow-07}.
\begin{lemma}\label{lemma:BakerPowers-best}
If $\gamma>\Gamma$, then there exists $\alpha'>0$
such that for all $0<\alpha<\alpha'$, the strong triangle inequality fails for
the triangle with angles $\alpha$, $\alpha$, and $\gamma$.
\end{lemma}

\begin{lemma}\label{lemma:underB}
Let $\gamma\in(\Gamma,B)$. Then the set of values of $\alpha$ for which
$\alpha$ is negative-positive is an open interval
$(0,i_\gamma)$, and $z$ is defined, continuous, and decreasing on this interval.
Furthermore, the region $N_\gamma$ is exactly the region under the function
$z$ on the interval $(0,i_\gamma)$.
\end{lemma}

\emph{See Figure~\ref{fig:1-2} for illustration.}

\begin{figure}
\includegraphics{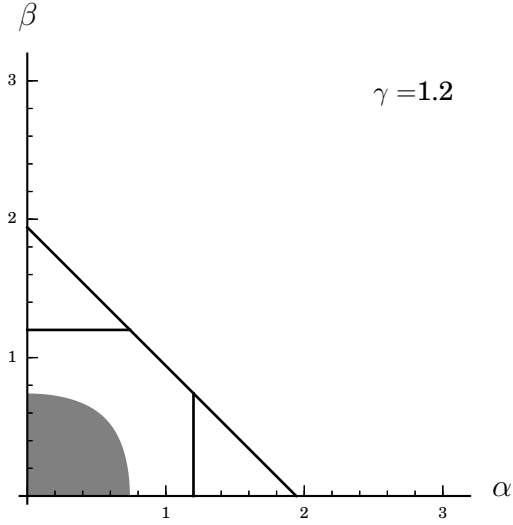}
\caption{For $\gamma=1.2$ here, the shaded region is $N_\gamma$, and the unshaded is $P_\gamma$. The boundary is $Z_\gamma$. The horizontal and vertical line segments represent $\beta=\gamma$ and $\alpha=\gamma$, beyond which it is guaranteed that the strong triangle inequality holds. The diagonal $\beta=\pi-\gamma-\alpha$ represents Euclidean triangles.\label{fig:1-2}}
\end{figure}

\begin{proof}
First note that the condition on $\gamma$ implies that every $\alpha$ is either
all-positive or negative-positive. Indeed, any other type of $\alpha$ would
give rise to a sequence of points in $N_\gamma$ converging to a point in
$\overline{P}_\gamma$. Let $C$ be the set of values of $\alpha$, for which $\alpha$ is
negative-positive. Clearly $z$ is defined on $C$.

By Lemma~\ref{lemma:BakerPowers-best}, there exists $\alpha'>0$
such that for all $0<\alpha<\alpha'$, $f(\alpha,\alpha,\gamma)<0$. This also
means that for all $0<\alpha<\alpha'$, $\alpha\in C$.

Since $\gamma < B$ it follows that the diagonal $\alpha + \beta + \gamma = \pi$ lies in $\overline{P}_\gamma$. Moreover, by Proposition~\ref{proposition:greatest}, the appropriate vertical line segment $\alpha = \gamma$  belongs to  $\overline{P}_\gamma$. Consider the open segment going right from the point $(\alpha'/2,\alpha'/2)$ and ending at the point $(x, \alpha'/2)$ such that $(x,\alpha'/2, \gamma) \in \overline{F} \setminus F$.  By Corollary~\ref{corollary:nojump},
we can not have the open segment entirely in $N_\gamma$. So there exists
$\alpha'/2<\alpha_0<\gamma$ with $(\alpha_0,\alpha'/2)\in Z_\gamma$. We also
have that $(0,\alpha_0)\subseteq C$.

By Lemma~\ref{lemma:formula}, $z$ is continuous on $(0,\alpha_0)$. Injective
continuous functions are monotone, and by symmetry again, $z$ must be monotone
decreasing on $(0,\alpha_0)$. The portion of $z$ on $(0,\alpha'/2)$ is
``copied'' to the portion after $\alpha_0$, so there exists $\alpha_1>\alpha_0$
such that $(0,\alpha_1)\subseteq C$, and $z$ is continuous, monotone decreasing
on $(0,\alpha_1)$, and $\lim_{\alpha\to\alpha_1^-}z(\alpha)=0$.

We claim that in fact $(0,\alpha_1)=C$. Suppose not, and there exists
$\alpha_2>\alpha_1$ with $\alpha_2\in C$. For all
$0<\beta_2<\min\{z(\alpha_2),\alpha'/2\}$, the horizontal line $\beta=\beta_2$
contains only one point from $Z_\gamma$. That implies that in fact the entire
open line segment between $(0,\beta_2)$ to
$(\min\{\pi-\gamma-\beta_2,\gamma\},\beta_2)$ lies in $N_\gamma\cup Z_\gamma$.
Thereby, we could construct a sequence in $N_\gamma\cup Z_\gamma$ converging to a point in
$\overline{P}_\gamma$ contrary to Corollary~\ref{corollary:nojump}. So the first part of the statement holds with
$i_\gamma=\alpha_1$.

The second part of the statement is obvious after the first part, which is
necessary to show that there is a well-defined region under the function on the
interval $(0,i_\gamma)$.
\end{proof}

For the actual computations, we will need to numerically compute the value of
$i_\gamma$. Since $i_\gamma=\lim_{\alpha\to 0^+} z(\alpha)$, and since $z$
remains continuous even if we extend the function by its formula for
$\alpha=0$, it is easy to compute its value. In fact it turns out that it has a
relatively simple formal expression:
\[
i_\gamma=\cos^{-1}\left(\frac{(\sin\gamma-1)^2+\cos\gamma}{2\sin\gamma-\cos\gamma-1}\right).
\]

Now we will start to work on the more difficult case when $\gamma\in(B,\pi/2)$.
First we need two technical lemmas.

\begin{lemma}\label{lemma:decreasing}
$f(\alpha,\beta,\gamma)$ is monotone decreasing in
$\gamma$.\label{lemma:decreasing_gamma}
\end{lemma}

\begin{proof}
Let
\[
f_1(\alpha,\beta,\gamma)=\frac{\cos\beta\cos\gamma+\cos\alpha}{\sin\gamma},\qquad
f_2(\alpha,\beta,\gamma)=\frac{\cos\alpha\cos\beta+\cos\gamma}
{\cos\gamma+1-\sin\gamma}-1.
\]
Then
\[
f(\alpha,\beta,\gamma)=\cos^2\beta+[f_1(\alpha,\beta,\gamma)]^2-[f_2(\alpha,\beta,\gamma)]^2.
\]
Simple computations show
\begin{align*}
\frac{\partial f_1}{\partial\gamma}
&=\frac{-\cos\beta\sin^2\beta-(\cos\beta\cos\gamma+\cos\alpha)\cos\gamma}
{\sin^2\gamma}<0\\
\frac{\partial f_2}{\partial\gamma}&=\frac
{1-\sin\gamma+\cos\alpha\cos\beta(\sin\gamma+\cos\gamma)}
{(\cos\gamma+1-\sin\gamma)^2}>0.
\end{align*}
Since for all $(\alpha,\beta,\gamma)\in F$, clearly $f_1>0$, and by
Lemma~\ref{lemma:nonnegative}, $f_2>0$, we get that
\[
\frac{\partial f}{\partial\gamma}=2f_1\frac{\partial f_1}{\partial\gamma}-2f_2\frac{\partial f_2}{\partial\gamma}<0.
\]
\end{proof}

\begin{lemma}\label{lemma:isosceles}
Let $\gamma\in[B,\pi/2)$. Then all isosceles triangles with angles
$\alpha$, $\alpha$, and $\gamma$ fail the strong triangle inequality.
Furthermore, these triangles fail with inequality, that is,
$(\alpha,\alpha,\gamma)\in N_\gamma$.
\end{lemma}

\begin{proof}
We will use Lemma~\ref{lemma:BakerPowers} with $\alpha=\beta$ and $\gamma=B$.
In that case, $\cos\gamma=7/25$ and $\sin\gamma=24/25$; also $\cosh
h=\cos\alpha/\sin(B/2)=\frac{5}{3}\cos\alpha$. It is elementary to see that to
satisfy the inequality of the lemma, even with equality, $\cos\alpha\leq 3/5$
is necessary, so $\alpha,\beta\geq\cos^{-1}(3/5)$, and then
$\alpha+\beta+\gamma\geq2\cos^{-1}(3/5)+\tan^{-1}(24/7)=\pi$.

So if $\gamma=B$, then all
$(\alpha,\alpha)\in N_\gamma$, and by
Lemma~\ref{lemma:decreasing}, this remains true for $\gamma>B$.
\end{proof}

\begin{lemma}\label{lemma:overB}
Let $\gamma\in[B,\pi/2)$.  Then the set of values of $\alpha$ for which
$\alpha$ is negative-positive is the union of two open intervals
$(0,e_\gamma)$ and $(\pi-\gamma-e_\gamma,i_\gamma)$, and $z$ is continuous and
decreasing on these intervals. No value $\alpha$ is
positive-negative. Furthermore, the region $N_\gamma$ is the region
under $z$ and under the line $\alpha+\beta+\gamma=\pi$.
\end{lemma}

\emph{See Figure~\ref{fig:1-3} for illustration.}

\begin{figure}
\includegraphics{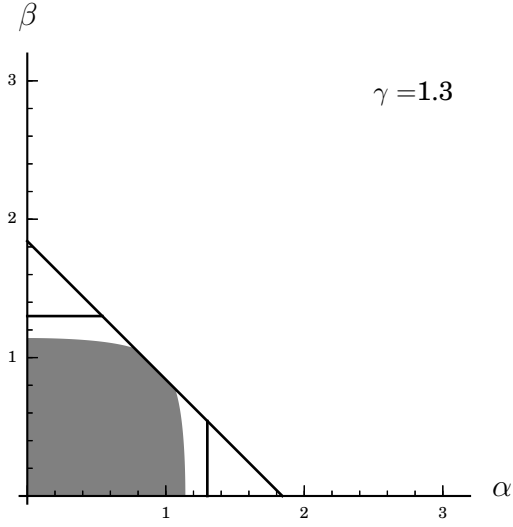}
\caption{Similar to Figure~\ref{fig:1-2}, the shaded region represents $N_\gamma$, this time for $\gamma=1.3$. See caption of Figure~\ref{fig:1-2} for additional explanation of features.\label{fig:1-3}}
\end{figure}

\begin{proof}
Let $\gamma\in[B,\pi/2)$. By \cite{Bai-Ban-97}, there exists
$(\alpha,\beta)\not\in\overline{P}_\gamma$ on the diagonal $\alpha+\beta+\gamma=\pi$.
It is implicit in \cite{Fai-Pow-Sah-13} that the set $\{(\alpha,\beta)\not\in
\overline{P}_\gamma:\alpha+\beta+\gamma=\pi\}$ is a closed line segment of the line
$\alpha+\beta+\gamma=\pi$, and the endpoints of this line segment are the only
points of $\overline{Z}_\gamma$ of the line. Also, this line segment is symmetric in
$\alpha$ and $\beta$. Let the two endpoints of the line segment have
coordinates $(e_\gamma,\pi-\gamma-e_\gamma)$, and
$(\pi-\gamma-e_\gamma,e_\gamma)$.

By Lemma~\ref{lemma:isosceles}, the open line segment from $(0,0)$ to
$(\frac{\pi-\gamma}{2},\frac{\pi-\gamma}{2})$ is entirely in $N_\gamma$. Let
\[
T=\{(\alpha,\beta):e_\gamma\leq\alpha,\beta\leq\pi-\gamma-e_\gamma\text{ and
}\alpha+\beta+\gamma<\pi\}.
\]
We will show that $T\subseteq N_\gamma\cup Z_\gamma$.
Indeed, suppose a point $(\alpha,\beta)$ in the interior of $T$ belongs to
$P_\gamma$. Without loss of generality $\alpha>\beta$. Then
there are $\alpha_0<\alpha<\alpha_1$ with $(\alpha_0,\beta),(\alpha_1,\beta)\in
T\cap N_\gamma$, and so by continuity, there are $\alpha_0'$ and
$\alpha_1'$ with $(\alpha_0',\beta),(\alpha_1',\beta)\in T\cap Z_\gamma$,
contradicting the fact that $Z_\gamma$ is a function. The statement for the
boundary of $T$ follows from Corollary~\ref{corollary:nojump}.

If $\alpha<e_\gamma$, then $\alpha$ is negative-positive. This is because
$(\alpha,\alpha)\in N_\gamma$ and $(\alpha,\min\{\pi-\gamma-\alpha,\gamma\})\in \overline{P}_\gamma$. So
$z_\gamma(\alpha)$ is defined on $(0,e_\gamma)$, and therefore it is continuous
on this interval.

Now we will show that $\lim_{\alpha\to e_\gamma^-}z(\alpha)=\pi-\gamma-e_\gamma$.
If this is not true, there is $\epsilon>0$ and a sequence
$\alpha_1,\alpha_2,\ldots$ with $\alpha_n\to e_\gamma$ such that
$z(\alpha_n)<\pi-\gamma-e_\gamma-\epsilon$. Let
$\beta_n=\pi-\gamma-e_\gamma-\epsilon/2$ (a constant sequence). Now the sequence $(\alpha_n,\beta_n)$
converges to the point $(e_\gamma,\pi-\gamma-e_\gamma-\epsilon/2)$, so a
sequence of points in $P_\gamma$, converges to a point in $N_\gamma\cup
Z_\gamma$. The only way this can happen if
$(e_\gamma,\pi-\gamma-e_\gamma-\epsilon/2)\in Z_\gamma$. But the argument can
be repeated with $\epsilon/3$ instead of $\epsilon/2$, so
$(e_\gamma,\pi-\gamma-e_\gamma-\epsilon/3)\in Z_\gamma$, and this contradicts
the fact that $Z_\gamma$ is a function.

Since $z(\alpha)$ is continuous and bijective on $(0,e_\gamma)$, it is
monotone. We will show it must be decreasing. First we note that for
$\gamma=B$, $z$ is clearly decreasing, because in that case
$e_\gamma=\frac{\pi-\gamma}{2}$, and by symmetry, the function is ``copied
over'' to the interval $(e_\gamma,i_\gamma)$, so it can not be increasing and
bijective. Then, since $f$ is continuous, $z_\gamma(\alpha)$ is continuous in
$\gamma$, so if $z_\gamma(\alpha_1)>z_\gamma(\alpha_2)$ for some
$\alpha_1<\alpha_2$ and $z_{\gamma'}(\alpha_1)<z_{\gamma'}(\alpha_2)$ for some
$\gamma'>\gamma$, then by the Intermediate Value Theorem, there is a
$\gamma<\gamma_0<\gamma'$ for which
$z_{\gamma_0}(\alpha_1)=z_{\gamma_0}(\alpha_2)$, a contradiction. Informally
speaking, the function $z$ can not flip its monotonicity without failing injectivity at some point.

We have already seen that $\alpha$ is negative-positive on $(0,e_\gamma)$. By
the fact that $z$ is decreasing on this interval, it is implied that $\alpha$
is all-negative on $[e_\gamma,\pi-\gamma-e_\gamma]$, and $\alpha$ is again
negative-positive on $(\pi-\gamma-e_\gamma,i_\gamma)$. Finally, $\alpha$ is
all-positive on $[i_\gamma,\gamma)$.

The last statement of the lemma is now clear.
\end{proof}

For the actual computations, we will need the value of $e_\gamma$. From
\cite{Fai-Pow-Sah-13}, which describes the equality case for Euclidean
geometry, we know that $e_\gamma$ is the value of $\alpha$ for which
\[
\tan\left(\frac\alpha2\right)+\tan\left(\frac\beta2\right)=1,
\]
and since the triangle is Euclidean, we have $\alpha/2+\beta/2=\pi/2-\gamma/2$.
These equations yield two symmetric solutions for $\alpha$ and $\beta$; by our
choice in the lemma, we need the smaller of these. We conclude
\[
e_\gamma=2\tan^{-1}\left(\frac{1}{2}-\sqrt{\tan\left(\frac{\gamma}{2}\right)-\frac{3}{4}}\right).
\]
For the proof of the next result we let
\[ \overline{S} = \{(\alpha, \beta, \gamma) \in F : f(\alpha, \beta, \gamma) \leq 0\}\]
and note that $\overline{S}$ is the set of points in $F$ where the strong triangle inequality fails.

\begin{theorem}\label{thm:probability}
The probability that the strong triangle inequality holds is
\begin{multline*}
\frac{7}{8}-
\frac{6}{\pi^3}
\bigg(
\int_{\Gamma}^{B}\int_0^{i_\gamma} z_\gamma(\alpha)\,d\alpha d\gamma+\\
\int_{B}^{\pi/2}\frac{(\pi-\gamma-2e_\gamma)^2}{2}-e_\gamma^2+
2\int_0^{e_\gamma} z_\gamma(\alpha)\,d\alpha d\gamma
\bigg).
\end{multline*}
\end{theorem}

\begin{proof}
We break up the integral
\begin{equation}\label{eq:vol}
\int_\Gamma^{\pi/2}\mu(N_\gamma)\ d\gamma
\end{equation}
over two intervals: $(\Gamma,B)$ and
$(B,\pi/2)$. By Lemma~\ref{lemma:underB}, in the former interval, $N_\gamma$ is
the region under the function $z_\gamma$. So if $\gamma\in(\Gamma,B)$, then
$\int_\Gamma^B \mu(N_\gamma)\ d\gamma=\int_\Gamma^B\int_0^{i_\gamma}
z_\gamma(\alpha)\,d\alpha\ d\gamma$. If $\gamma\in(B,\pi/2)$, then, by
Lemma~\ref{lemma:overB} and symmetry,
\[
\int_B^{\pi/2}\mu(N_\gamma)\ d\gamma=\int_B^{\pi/2}\left(2\int_0^{e_\gamma}z_\gamma(\alpha)\,d\alpha-e_\gamma^2+(\pi-\gamma-2e_\gamma)^2/2\right)\ d\gamma.
\]
Thus,
\begin{multline*}
\vol(\overline{S})=
\int_{\Gamma}^{B}\int_0^{i_\gamma} z_\gamma(\alpha)\,d\alpha d\gamma\ +\\
\int_{B}^{\pi/2}\left(\frac{(\pi-\gamma-2e_\gamma)^2}{2}-e_\gamma^2+
2\int_0^{e_\gamma} z_\gamma(\alpha)\,d\alpha \right) d\gamma.
\end{multline*}
By Proposition~\ref{proposition:obtuse}, the strong triangle inequality does not hold if $\gamma \geq \frac{\pi}{2}$. The volume of the tetrahedron for $\gamma\geq\pi/2$ is $\pi^3/48$. Since the volume of the tetrahedron with $\gamma \geq 0$ is $\pi^3/6$ it follows that the required probability is
\[
1-\left(\frac{\vol(\overline S)}{\pi^3/6} + \frac{1}{8}\right),
\]
and the formula follows.
\end{proof}

\section{Theoretical error estimates}

We are almost ready to use our favorite computer algebra system to compute the
actual number. However, numerical integration will not guarantee accurate
results in general. To make sure that we can (theoretically) control the error
of computation, we need one more theorem.

\begin{theorem}
The volume of $\overline{S}$ may be approximated by arbitrary precision.
More precisely, for all $\epsilon>0$ there is an algorithm to compute a
numerical upper bound $M$ and a lower bound $m$ such that $m <\vol(\overline{S}) < M$
and $M-m<\epsilon$.
\end{theorem}

\begin{proof}
Lemma~\ref{lemma:decreasing_gamma} implies that in (\ref{eq:vol}) we
integrate a monotone increasing function, because $\mu(N_\gamma)$ is the
measure of the level set of $f$ at
$\gamma$.  Recall that for a monotone decreasing (respectively, increasing)
function, the left Riemann sum overestimates (underestimates) the integral, and
the right Riemann sum underestimates (overestimates) it. That is, it is
possible to know how precise the the numerical estimate is, and if necessary,
it is possible to repeat the computation with higher resolution.

In the actual computation given by Theorem~\ref{thm:probability}, both terms
in the parenthesis involve computations of integrals of monotone
functions, and the inner integrals in those terms are also computing integrals
of monotone functions. So, in essence, the numerical computation
involves the integration of a monotone increasing function, whose values may be
approximated at arbitrary precision.
\end{proof}

\section{Conclusion}

We can now use Theorem~\ref{thm:probability} and the computer algebra system Sage to get the following result.

\begin{corollary}
Under the assumption that $\alpha, \beta, \gamma$ can be chosen uniformly in the interval $(0,\pi)$  and $\alpha+\beta+\gamma<\pi$, the strong triangle inequality $a + b > c + h$ holds approximately 78.67\% of the time.
\end{corollary}

Since we know that the strong triangle inequality fails when $\gamma \geq \pi/2$, we could restrict our attention to triangles where $\gamma < \pi/2$. In this case, the inequality $a + b > c + h$ holds approximately 90\% of the time. For the Euclidean case, where $\alpha + \beta + \gamma = \pi$ and $\gamma < \pi/2$, it was shown in \cite{Fai-Pow-Sah-13} that the strong triangle inequality holds approximately 92\% of the time. Since the calculations in this paper involved volumes and the calculations in \cite{Fai-Pow-Sah-13} involved areas, it is hard to directly compare the hyperbolic and Euclidean probabilities of the strong triangle inequality. We can say, however, that in both planes the strong triangle inequality is  likely to hold.

\appendix

\section{Sage code}

The following code will visualize the value $a+b-c-h$ (referred as
``strength'') of a labelled triangle depending on the angles.  It generates
$2000$ pictures (or ``frames''), and each frame will correspond to a fixed
value of the angle $\gamma$, which grows throughout the frames from $0$ to
$\pi/2$. The number of frames is defined with the variable \texttt{number}.
For each frame, the strength is indicated for the angles $\alpha$, $\beta$, as
the color of a point in the $(\alpha,\beta)$ coordinate system.  Small positive
strength is indicated by blue colors, high positive strength is indicated by
red colors. The contours are changing from $0$ to $1$.  Negative strength will
be simply the darkest blue. To make the frames more informative, this darkest
blue color may be replaced by a distinctive color outside of Sage (e.g.\ using
Imagemagick).  A black square on the bottom left corner indicates the points
for which $\gamma$ is the greatest angle. Outside of this square, the strength
is proven to be positive.  The pictures are saved as numbered \texttt{png}
files.

\scriptsize
\begin{verbatim}
sage: def strength(al,be,ga): #this is a+b-c-h
...       cha=(cos(be)*cos(ga)+cos(al))/(sin(be)*sin(ga))
...       chb=(cos(al)*cos(ga)+cos(be))/(sin(al)*sin(ga))
...       chc=(cos(al)*cos(be)+cos(ga))/(sin(al)*sin(be))
...       a=arccosh(cha)
...       b=arccosh(chb)
...       c=arccosh(chc)
...       shb=sqrt(chb^2-1)
...       shh=shb*sin(al)
...       h=arcsinh(shh)
...       expression=a+b-c-h
...       return expression
...       
sage: def defect(al,be,ga): return pi-al-be-ga
...
sage: var("al be ga")
sage: con=[]
sage: for i in xrange(50): con.append(i/50)
sage: map=sage.plot.colors.get_cmap('coolwarm')
sage: number=2000
sage: for i in xrange(number):
...       gamma=(i+1)*(pi/2)/(number)
...       p=contour_plot(strength(al,be,ga=gamma),(al,0,pi),(be,0,pi),
...           contours=con,cmap=map,plot_points=1000,
...           figsize=[10,10],region=defect(al,be,ga=gamma))
...       p+=line([(0,pi-gamma),(pi-gamma,0)],color='black')
...       p+=line([(0,gamma),(min(pi-2*gamma,gamma),gamma)],color='black')
...       p+=line([(gamma,0),(gamma,min(pi-2*gamma,gamma))],color='black')
...       p+=text("$\\gamma=$"+str(float(gamma)),(2.5,3),
...           vertical_alignment='top',horizontal_alignment='left')
...       p.save('hyper'+str(i).zfill(4)+'.png')

\end{verbatim}

\normalsize
A video generated by this code can be found at
\url{http://www.math.louisville.edu/~biro/movies/sti.mp4}. In this video,
negative strength is represented by the color green. To generate the video, the
following commands were executed in Bash (Linux Mint 17.1, ImageMagick and
libav-tools installed). The reason of cropping in the second line is that the
default mp4 encoder for avconv (libx264) requires even height and width.

\scriptsize
\begin{verbatim}
for i in hyper*.png; do convert $i -fill green -opaque "#3b4cc0" x$i; done
avconv -i xhyper%04d.png -r 25 -vf "crop=2*trunc(iw/2):2*trunc(ih/2):0:0" -b:v 500k sti.mp4
\end{verbatim}

\normalsize
The following code performs the numerical computation of the integral. We are
trying to follow the paper as close as possible, including notations. Note that
the numerical integration is performed by Gaussian quadrature, so error bounds
are not guaranteed in this code. We use the \texttt{mpmath} package and we
store 100 decimal digits.

\scriptsize
\begin{verbatim}
sage: from mpmath import *
sage: mp.dps=100
sage: Gamma=findroot(lambda x: -1-cos(x)+sin(x)+sin(x/2)*sin(x),1.15)
sage: Beta=atan(24/7)
...       
sage: def i(gamma):
...       return acos(((sin(gamma)-1)^2+cos(gamma))/(2*sin(gamma)-cos(gamma)-1))
...       #return z(gamma,0) #This should give the same result
...
sage: def e(gamma):
...       D=tan(gamma/2)-3/4
...       if D<0:
...           sol=1/2
...       else:
...           sol=1/2-sqrt(D)
...       return 2*atan(sol)
...           
sage: def z(gamma,alpha):
...       denominator=cos(gamma)+1-sin(gamma)
...       a=csc(gamma)^2-(cos(alpha)/denominator)^2
...       b=cos(alpha)*(cos(gamma)+1)/sin(gamma)^2
...       c=(cos(alpha)/sin(gamma))^2-((sin(gamma)-1)/denominator)^2
...       d=b^2-4*a*c
...       if d>=0:
...           sol=(-b-sqrt(d))/(2*a)
...       else:
...           sol=-b/(2*a)
...       if sol>1 or sol<-1:
...           result=0
...       else:
...           result=min(acos(sol),pi-alpha-gamma)
...       return result
...       
sage: f = lambda gamma: quad(lambda alpha: z(gamma,alpha),[0,i(gamma)])
sage: g = lambda gamma: (pi-gamma-2*e(gamma))^2/2-e(gamma)^2+
....      2*quad(lambda alpha: z(gamma,alpha),[0,e(gamma)])
sage: int1=quad(f,[Gamma,Beta])
sage: int2=quad(g,[Beta,pi/2])
sage: print "Probability:", 7/8-(6/pi^3)*(int1+int2)
\end{verbatim}

\bibliographystyle{amsplain}
\bibliography{bib}

\end{document}